\documentclass[11pt,reqno]{amsart}
\usepackage{amsmath}
\usepackage{amssymb}
\usepackage{tikz}
\usepackage{color}
\usepackage{xcolor}
\usepackage{float}
\usepackage{cite}
\usepackage{geometry}
 \theoremstyle{definition}

 \newtheorem{theorem}{Theorem}[section]%
 
 \newtheorem{lemma}[theorem]{Lemma}

\numberwithin{equation}{section}

\theoremstyle{remark}

\newcommand{\abs}[1]{\left|#1\right|}

\newcommand{\be}{\begin{equation}}
\newcommand{\ee}{\end{equation}}

\newcommand\R{\mathbb R}
\newcommand\N{\mathbb N}
\newcommand\Z{\mathbb Z}

\newcommand{\T}{\mathbb{T}}
\newcommand{\Sn}{\mathbb{S}^n}
\newcommand{\Tn}{\mathbb{T}^n}

\newcommand\cube{\mathcal C}

\newcommand\Wp{\mathcal{W}_p}
\newcommand\Wt{\mathcal{W}_2}
\newcommand\Wo{\mathcal{W}_1}
\newcommand\dwpp{d_{\Wp}^p}
\newcommand\dwp{d_{\Wp}}
\newcommand\dwt{d_{\Wt}}

\newcommand{\isom}{\mathrm{Isom}}

\newcommand{\prob}{\mathcal{P}}

\newcommand{\WpSg}{\Wp(\Sn)}

\newcommand{\WpTd}{\Wp(\Tn)}
\newcommand{\WoTd}{\Wo(\Tn)}
\newcommand{\WtTd}{\Wt(\Tn)}

\newcommand{\ler}[1]{\left( #1 \right)}

\newcommand{\dd}{\mathrm{d}}

\newcommand{\inner}[2]{\left< #1,#2 \right>}

\newcommand{\norm}[1]{\left|\left|#1\right|\right|}

\newcommand{\potmu}{\mathcal{T}_\mu^p}
\newcommand{\potmuo}{\mathcal{T}_\mu^1}
\newcommand{\potmut}{\mathcal{T}_\mu^2}

\newcommand{\potnu}{\mathcal{T}_\nu^p}

\newcommand{\potphimu}{\mathcal{T}_{\Phi(\mu)}^p}

\newcommand{\proj}{\mathfrak{p}}

\title[Isometric rigidity of Wasserstein tori and spheres]{Isometric rigidity of Wasserstein tori and spheres}                                                                                            

\author[Gy\"orgy P\'al Geh\'er]{Gy\"orgy P\'al Geh\'er}
\address{Gy\"orgy P\'al Geh\'er, Department of Mathematics and Statistics\\ University of Reading\\ Whiteknights\\ P.O.
Box 220\\ Reading RG6 6AX\\ United Kingdom}
\email{gehergyuri@gmail.com}

\author[Tam\'as Titkos]{Tam\'as Titkos}
\address{Tam\'as Titkos, Alfr\'ed R\'enyi Institute of Mathematics\\ Re\'altanoda u. 13-15.\\
Budapest H-1053\\ Hungary\\ and BBS University of Applied Sciences\\ Alkotm\'any u. 9.\\
Budapest H-1054\\ Hungary}

\email{titkos.tamas@renyi.hu}

\author[D\'aniel Virosztek]{D\'aniel Virosztek}
\address{D\'aniel Virosztek, Alfr\'ed R\'enyi Institute of Mathematics\\ Re\'altanoda u. 13-15.\\Budapest H-1053\\ Hungary}

\email{virosztek.daniel@renyi.hu}

\begin{document}


\thanks{Geh\'er was supported by the Leverhulme Trust Early Career Fellowship (ECF-2018-125), and also by the Hungarian National Research, Development and Innovation Office (Grant no. K115383); Titkos was supported by the Hungarian National Research, Development and Innovation Office - NKFIH (grant no. PD128374 and grant no. K115383) and by the János Bolyai Research Scholarship of the Hungarian Academy of Sciences; Virosztek was supported by the Momentum program of the Hungarian Academy of Sciences under grant agreement no. LP2021-15/2021, and partially supported by the Hungarian National Research, Development and Innovation Office - NKFIH (grants no. K124152 and no. KH129601).}

\begin{abstract} We prove isometric rigidity for $p$-Wasserstein spaces over finite-dimensional tori and spheres for all $p$. We present a unified approach to proving rigidity that relies on the robust method of recovering measures from their Wasserstein potentials.
\end{abstract}

\maketitle
\section{Introduction}

Given a metric space $(X,r)$ and a subset $\mathcal{S}\subseteq\mathcal{P}(X)$ of all probability measures, 
one can endow $\mathcal{S}$ with various metrics, depending on what kind of measurement is suitable for the problem under consideration. Here we mention three examples.
\begin{itemize} 
\item[-] The \emph{Kolmogorov-Smirnov metric} $d_{KS}$ on $\mathcal{S}=\mathcal{P}(\mathbb{R})$ is frequently used in statistics to compare a
sample with a reference probability distribution.

\item[-] The \emph{Lévy-Prokhorov metric} $d_{LP}$ plays an important theoretical role in several limit theorems in probability theory. In this case $(X,\varrho)$ is a complete separable metric space and $\mathcal{S}=\mathcal{P}(X)$.

\item[-] The \emph{quadratic Wasserstein metric} $d_{\mathcal{W}_2}$ turned out to be very effective in a wide range of AI applications including pattern recognition and image processing problems. In these applications $(X,r)$ is typically the $n$-dimensional Euclidean space and $\mathcal{S}$ is the collection of all Borel probability measures with finite second moment.
\end{itemize}
In recent years, there has been a considerable interest in the characterization of the above mentioned (and many other) metric spaces of measures, see e.g. \cite{bertrand-kloeckner-2016,DolinarKuzma,DolinarMolnar,Kuiper,LP,JMAA,TAMS,HIL,Kloeckner-2010,Levy,S-R,titkoskissgraf,Virosztek}. In most cases, it turned out that isometries of $\mathcal{S}$ are strongly related to self-maps of the underlying space $X$. Concerning the Kolmogorov-Smirnov distance, Dolinar and Moln\'ar showed in \cite{DolinarMolnar} that there is a one-to-one correspondence between all isometries of $(\mathcal{P}(X),d_{KS})$ and all homeomorphisms of the real line. Concerning the Lévy-Prokhorov metric, the first and the second author showed in \cite{LP} that $\mathcal{P}(X)$ endowed with $d_{LP}$ is more rigid, assuming that $X$ is real separable Banach space: a self-map of $X$ induces an isometry on $(\mathcal{P}(X),d_{LP})$ if and only if it is itself an isometry. In fact, the isometry group of $(X,\|\cdot\|)$ and the isometry group of $(\mathcal{P}(X),d_{LP})$ are isomorphic. This phenomenon is called \emph{isometric rigidity}. The third example above is even more peculiar. Kloeckner showed in \cite{Kloeckner-2010} that the isometry group of the quadratic Wasserstein space over $\mathbb{R}^n$ is much larger than the isometry group of $\mathbb{R}^n$. 

The most recent results which are related to our studies have been presented in \cite{S-R}: Santos-Rodríguez proved isometric rigidity for a very broad class of manifolds. More precisely, he showed that the isometry group of a quadratic Wasserstein space over a closed Riemannian manifold with strictly positive sectional curvature is isomorphic to the isometry group of the underlying manifold. Furthermore, for compact rank one symmetric spaces (CROSSes), Santos-Rodríguez was able to prove isometric rigidity not only for the quadratic case, but for general $p$-Wasserstein spaces with $1<p<\infty$.\\

As the results of \cite{Kloeckner-2010} and \cite{S-R} already indicate, isometric rigidity of $p$-Wasserstein spaces depends in an interesting way both on some characteristics of the underlying space $X$ and on the value of $p$. To stress this phenomenon, we briefly mention one more special case, for more details see \cite{TAMS}: the $p$-Wasserstein space over $\mathbb{R}$ is isometrically rigid if and only if $p\neq2$, while the $p$-Wasserstein space over $[0,1]$ is isometrically rigid if and only if $p\neq1$.\\

This paper aims to offer a unified approach for two compact manifolds as underlying space: the n-dimensional torus and the n-dimensional sphere, regardless of what the value of $p$ is. The key idea -- which could be applicable in other settings as well -- is that all measures can be fully recovered from a function, which we call Wasserstein potential. The main results are \textbf{Theorem \ref{thm:WpTd}} and \textbf{Theorem \ref{thm:WpSg}}, where we prove for all $p\geq 1$ that the isometry group of the $p$-Wasserstein space on the $n$-dimensional torus is isomorphic to the isometry group of the torus, and similarly, the isometry group of the $p$-Wasserstein space over the $n$-dimensional sphere is isomorphic to the isometry group of the sphere. The latter result has been partially covered by Santos-Rodríguez in \cite{S-R}, as the sphere is a compact rank one symmetric space. However, the method presented in \cite{S-R} cannot be extended to the case $p=1$, as $1$-Wasserstein spaces have a more flexible structure: the optimal transport plan between measures is not unique, let
alone the geodesic curve.

\section{The Wasserstein potential of measures}

In this section, we collect all notions which are necessary to our investigations. We also demonstrate via a simple example how useful the Wasserstein potential can be to identify measures and to prove isometric rigidity.\\

Let $(X,r)$ be a complete and separable metric space. We denote by $\mathcal{P}(X)$ the collection of all Borel probability measures on $X$, and by $\mathcal{F}(X)$ the set of all finitely supported measures. Given a measure $\mu\in\mathcal{P}(X)$, the support $S(\mu)$ is the set of all points $x\in X$ for which every open neighbourhood of $x$ has positive measure with respect to $\mu$. A Borel probability measure $\pi$ on $X \times X$ is said to be a \emph{coupling} for $\mu,\nu\in\prob(X)$ if
\begin{equation}\pi\ler{A \times X}=\mu(A)\qquad\mbox{and}\qquad\pi\ler{X \times B}=\nu(B)
\end{equation}
for all Borel sets $A,B\subseteq X$. We denote set of all couplings by $\Pi(\mu,\nu)$. For any $1\leq p<\infty$ one can define the \emph{$p$-Wasserstein space} $\Wp(X)$ as the set of all $\mu\in\prob(X)$ that satisfy 
\begin{equation}\label{pp}\int_X r(x,\hat{x})^p~\mathrm{d}\mu(x)<\infty
\end{equation} for all $\hat{x}\in X$, endowed with the \emph{$p$-Wasserstein distance}
\begin{equation} \label{eq: wasser def}
\dwp\ler{\mu, \nu}:=\ler{\inf_{\pi \in \Pi(\mu, \nu)} \int_{X \times X} r(x,y)^p~\dd \pi(x,y)}^{1/p}.
\end{equation}
This distance measures the minimal effort required to transport one probability measure into another, when the cost of moving mass is the $p$-th power of the distance. For more details on optimal transport and Wasserstein spaces we refer the reader to the comprehensive textbooks of Santambrogio and Villani \cite{Santambrogio,Villani}. 

It is one of the important features of $p$-Wasserstein spaces that $\Wp(X)$ contains an isometric copy of $X$, as the distance between any two Dirac measures equals to the distance of their supporting points, i.e. $\Wp(\delta_x,\delta_y)=r(x,y)$. Furthermore, every measure belonging to $\Wp(X)$ can be approximated by convex combinations of Dirac measures, that is, $\mathcal{F}(X)$ is dense in $\Wp(X)$. (For more details see e.g. Example 6.3 and Theorem 6.16 in \cite{Villani}.)

In this paper we are interested in the structure of \emph{isometries}, that is, distance preserving bijections.  The symbol $\isom(\cdot)$ will always refer to the isometry group of the metric space in question. We denote the \emph{push-forward map} of an isometry $\psi\colon X \rightarrow X$ by
$\psi_\# \colon \Wp(X)\to\Wp(X)$: \begin{equation}\label{pushforward}
\big(\psi_\#(\mu)\big)(A)=\mu(\psi^{-1}[A])
 \end{equation}
for all $A\subseteq X$ and $\mu\in\Wp(X)$, where $\psi^{-1}[A]=\{x\in X\,|\, g(x)\in A\}$. If $p\geq1$, the push-forward operation
\begin{equation}
    \#:\isom(X)\to\isom(\Wp(X));\quad \psi\mapsto\psi_\#
\end{equation}
is an embedding (in fact, a group homomorphism). Those isometries which belong to the image of $\#$ are called \emph{trivial isometries}. We say that $\Wp(X)$ is \emph{isometrically rigid} if $\#$ is surjective. Now we introduce our key tool: for a given $\mu\in\Wp(X)$ the one variable function $\potmu:X\to\mathbb{R}$ defined by
\begin{equation}\label{potdef}
	\potmu(x):=\dwpp(\mu,\delta_x)=\int_{\Tn} r(x,y)^p ~\dd\mu(y).
\end{equation}
is called the \emph{Wasserstein potential of $\mu$}.
We expect that $\mu$ can be fully recovered from this function, and in particular that $\potmu=\potnu$ implies $\mu=\nu$.\\

In this  paper we are going to consider the torus and the sphere as underlying spaces. The symbol $\T^n$ stands for the $n$-dimensional torus, that is, the set $\R^n/\Z^n\simeq[-1/2,1/2)^n$ equipped with the usual metric
\begin{equation}\label{torusmetric}\varrho\ler{x,y}=\ler{\sum_{k=1}^n\abs{\ler{x_k-y_k}_{\text{mod }1}}^2}^{\frac{1}{2}},
\end{equation}
where $x=(x_1,\dots,x_n)\in\R^n$ and $y=(y_1,\dots,y_n)\in \R^n$. The antipodal of a point $(x_1,\dots,x_n)$ in the torus is $(x_1+1/2,\dots,x_n+1/2)$. We denote the unit sphere of $\R^{n+1}$ by the symbol $\Sn$, that is, $\Sn := \{x\in\R^{n+1}\colon \norm{x}=1\}$. We equip the unit sphere with the angular (or geodesic) distance: for $x,y\in\Sn$ the distance of $x$ and $y$ is
\begin{equation}\label{sphericalangle}
\sphericalangle(x,y) := \arccos\inner{x}{y}.
\end{equation}
We say that two points $x$ and $y$ are antipodal in the sphere if $y=-x$.  Adapting Gangbo's and Tudorascu's terminology in \cite{GT}, we will shortly refer to the $p$-Wasserstein spaces over $(\Tn,\varrho)$ and $(\Sn,\sphericalangle)$ as the $p$-Wasserstein torus and the $p$-Wasserstein sphere, respectively.\\

To conclude this chapter, we present one possible way of using Wasserstein potentials. This example will shed some light also on the difficulties that need to be overcome to obtain the desired result for all possible values of p. First of all, when working with the torus, it is a natural idea to borrow techniques from the theory of Fourier analysis.\\

\noindent Given a measure $\mu\in\Wp(\Tn)$, the potential function $\potmu$ can be written for all $x\in\Tn$ as
\begin{equation}\potmu(x)=\int_{\T^n} \varrho^p(x,y) ~\dd \mu(y)= \int_{\T^n}c_p(x-y) ~\dd \mu(y)= \ler{c_p * \mu}(x),
\end{equation}
where $c_p(x)=\ler{\sum_{k=1}^n x_k^2}^{\frac{p}{2}}$. Since the \emph{characters} of $\T^n$
$$
\varphi_j(x)=e^{2 \pi i j \cdot x} \qquad \ler{j=\ler{j_1,j_2, \dots,j_n} \in \Z^n}
$$
form an orthonormal basis of $L^2\ler{\T^n}$, we have
\begin{equation*}
\begin{split}
\hat{\mathcal{T}}_\mu^p(j)&=\inner{\potmu}{\varphi_j}
=\int_{\T^n} \potmu(x) e^{-2\pi i j \cdot x} \dd x\\
&=\int_{\T^n}\Big( \int_{\T^n}c_p(x-y) ~\dd \mu(y)\Big) e^{-2\pi i j \cdot x} ~\dd x\\
&=\int_{\T^n} \Big(\int_{\T^n}c_p(x-y) e^{-2\pi i j \cdot (x-y)}\dd x\Big) e^{-2\pi i j \cdot y}~\dd \mu(y)\\
&=\inner{c_p}{\varphi_j}\inner{\mu}{\varphi_j}=\hat{c}_p(j)\hat{\mu}(j).
\end{split}
\end{equation*}
In particular, if $\hat{c}_p(j)\neq 0$ for every $j \in \Z^n,$ then the measure can be recovered from the potential function. If $n=1$ and $p=2$ then the Fourier series of $c_2(x)=x^2$ does not vanish anywhere. More precisely, $\hat{c}_2(j)=\frac{(-1)^j}{2 j^2 \pi^2}$ for $j \neq 0$ and $\hat{c}_2(0)=\frac{1}{12}$. This means that $\potmu\equiv\potnu$ implies $\mu=\nu$ in this case. As we will see later, this implication automatically ensures that the Wasserstein space in question is isometrically rigid. 

Based on numerical computations, it seems that the Fourier transform $\hat{c}_p$ does not vanish anywhere if $n=1$ and $p>1.$ However, this is not the case for $n=1$ and $p=1.$ Indeed, $\hat{c}_1(j)=0$ for non-zero even $j$'s, $\hat{c}_1(0)=\frac{1}{4},$ and $\hat{c}_1(j)=-\frac{1}{j^2 \pi^2}$ for odd $j$'s. This does not mean that $\Wo \ler{\T}$ is not isometrically rigid, but we cannot prove it in such a simple way. For $n>1,$ the same holds true for $\Wt\ler{\T^n}$. The reason is that the summands of the quadratic cost function $c_2(x)=\sum_{k=1}^n x_k^2$ depend on only one variable, and hence $\hat{c}_2(j)=0$ whenever $j_k \neq 0$ for at least two indices. For example, for $n=2,$ we have  $\hat{c}_2 (0,0)=\frac{1}{6},$  $\hat{c}_2(0,j)=\hat{c}_2(j,0)=\frac{(-1)^j}{2 j^2 \pi^2}$ for  $j\neq 0$, and $\hat{c}_2(j_1,j_2)=0$ for $j_1, j_2 \neq 0$.\\

In what follows we develop a method that works for all $p\geq1$ and is suitable to prove that $\Wp(\Tn)$ and $\Wp(\Sn)$ are isometrically rigid. In fact, this method works in the $0<p<1$ case as well, but we decided to not include it in the main body. On the one hand, we have already proved in \cite{HIL} that $p$-Wasserstein spaces are all isometrically rigid if $0<p<1$, regardless of what the underlying space is. On the other hand, as the definition of the $p$-Wasserstein distance is slightly different in the $0<p<1$ case, we should add one more branch to all proofs, without any serious novelty.

\section{Isometric rigidity of the Wasserstein torus}

We start with a simple observation: the diameter of $\WpTd$ is $\sqrt{n}/2$ if $p\geq 1$, and this maximal distance is achieved if and only if the two measures are Dirac masses concentrated on antipodal points. This automatically implies that if $\Phi\in\isom(\Wp(\Tn)$ then the $\Phi$-image of a Dirac measure is again a Dirac measure, in fact, Dirac measures concentrated on antipodal points are mapped to Dirac measures which are concentrated on antipodal points. Since $\dwp(\delta_x,\delta_y)=\varrho(x,y)$, this implies that the map $\psi:\Tn\to\Tn$ defined by $\Phi(\delta_x)=\delta_{\psi(x)}$ is an isometry of $\Tn$. It is a known that any isometry $\psi$ of $\Tn$ can be written in the following form:
\begin{equation}\label{eq:TD-isom}
    \psi((x_1,x_2,\dots,x_n)) = (\varepsilon_1 x_{\sigma(1)},\varepsilon_2 x_{\sigma(2)},\dots,\varepsilon_n x_{\sigma(n)}) + (u_1,u_2,\dots,u_n)
\end{equation}
with a permutation $\sigma$, numbers $\varepsilon_1,\dots,\varepsilon_n\in\{-1,1\}$ and point $(u_1,u_2,\dots,u_n)\in\Tn$.

By the above observation, we have
$\potphimu(\psi(x)) = \potmu(x)$ for all $x\in\Sn$ which suggests that those properties of $\mu$ which are encoded in its potential function, will be carried over to $\Phi(\mu)$.

Before continuing, we need some new notations. Let $x\in\Tn$, $n\geq 2$, $j\in\{1,\dots,n\}$. We introduce the set $H(x,j):=\{(y_1,\dots,y_n)\in\T^n\colon y_j=x_j \}$, and we denote by $e^j$ the vector $(\delta_{j,1},\dots,\delta_{j,n})$, where $\delta_{j,j}=1$ and $\delta_{i,j}=0$ if $i\neq j$. The symbol $\check{x}_j\in\T^{n-1}$ stands for the point obtained by dropping the $j$th coordinate of $x\in\Tn$, and $\lambda_{n-1}$ denotes the normalised Haar measure of $\T^{n-1}$. We remark that $\T^{n-1}$ can be identified with $H(x,j)$ for any point $x\in\Tn$ and $j\in\{1,\dots,n\}$. For two points $x,y\in\Tn$ we denote by $B(x,y)$ the bisector of $x$ and $y$, i.e., $B(x,y)=\{z\in\Tn\,|\,\varrho(x,z)=\varrho(z,y)\}$. The following lemma is of key importance. We will use it later to estimate the measure of certain (carefully chosen) sets and points by means of the Wasserstein potential.

\begin{lemma}\label{lem:WpTd-pot-p}
Let $n\in\N$, $n\geq 2$, $p\geq 1$, $x\in\Tn$, $j\in\{1,\dots,n\}$, and $\mu\in\WpTd$. Then the following assertions hold:
\begin{itemize}
\item[(a)] If $p=1$, and $\varrho_{n-1}$ denotes the distance of $\mathbb{T}^{n-1}$ then
    \begin{align}\label{eq:WpTd-pot-p=1}
    &\lim_{s\to 0+}\frac{\potmuo(x+s\cdot e^j) - 2\potmuo(x) + \potmuo(x-s\cdot e^j)}{s}\nonumber \\
    & \hspace{2.5cm} =
    2\mu(\{x\}) - \int_{H(x+\frac{1}{2}e^j,j)} \ler{\frac{1}{4} + \varrho_{n-1}\ler{\check{x}_j,\check{y}_j}^2}^{-\frac{1}{2}} ~\dd\mu(y).
    \end{align}
    \item[(b)] If either $1<p<2$ or $p>2$, then
    \begin{align}\label{eq:WpTd-pot-p>1-not2}
    &\lim_{s\to 0+}\frac{\potmu(x+s\cdot e^j) - 2\potmu(x) + \potmu(x-s\cdot e^j)}{s} \nonumber \\ 
    & \hspace{2.5cm} = - p\int_{H(x+\frac{1}{2}e^j,j)} \ler{\frac{1}{4} + \varrho_{n-1}\ler{\check{x}_j,\check{y}_j}^2}^{\frac{p-2}{2}} ~\dd\mu(y).
    \end{align}
    \item[(c)] If $p=2$, then 
    \begin{align}\label{eq:WpTd-pot-p=2}
    &\lim_{s\to 0+}\frac{\potmut(x+s\cdot e^j) - 2\potmut(x) + \potmut(x-s\cdot e^j)}{s} =
    - 2\cdot\mu\ler{H\ler{x+\frac{1}{2}e^j,j}}.
    \end{align}
\end{itemize}
\end{lemma}

\begin{proof}

Since $p\geq 1$, we have the following:
\begin{align*}
    &\lim_{s\to 0+}\frac{\potmu(x+s\cdot e^j) - 2\potmu(x) + \potmu(x-s\cdot e^j)}{s} \\
   = &\lim_{s\to 0+} \bigg( \int_{\Tn\setminus\{x\}\setminus H(x+\frac{1}{2}e^j,j)} \frac{\varrho(x+s\cdot e^j,y)^p - 2 \varrho(x,y)^p + \varrho(x-s\cdot e^j,y)^p}{s}~\dd\mu(y)   \\
    & \hspace{3cm} + 2\mu(\{x\})\frac{s^p}{s} + 2\int_{H(x+\frac{1}{2}e^j,j)} \frac{\varrho(x+s\cdot e^j,y)^p - \varrho(x,y)^p}{s}~\dd\mu(y) \bigg).
\end{align*}First we obtain that the above two integrands are bounded. 
On the one hand, since $s\mapsto \varrho(x+s\cdot e^j,y)^p$ is differentiable at $s=0$ if $y\in \Tn\setminus\{x\}\setminus H(x+\frac{1}{2}e^j,j)$, the first integral converges to $0$ by Lebesgue's dominant convergence theorem.
On the other hand, if $y\in H(x+\frac{1}{2}e^j,j)$, then by elementary calculus we obtain
\begin{align*}
    &\lim_{s\to 0+} \frac{\varrho(x+s\cdot e^j,y)^p - \varrho(x,y)^p}{s} 
    = \left.\frac{d}{ds}\ler{\ler{\ler{\frac{1}{2}-s}^2 + \sum_{\substack{k=1 \\ k\neq j}}^n(x_k-y_k)^2}^{p/2}}\right|_{s=0} \\
    &= -\frac{p}{2} \ler{\ler{\frac{1}{2}}^2 + \sum_{\substack{k=1 \\ k\neq j}}^n(x_k-y_k)^2}^{\frac{p-2}{2}}
    = -\frac{p}{2} \ler{\frac{1}{4} + \varrho_{n-1}\ler{\check{x}_j,\check{y}_j}^2}^{\frac{p-2}{2}}.
\end{align*}
An application of Lebesgue's dominant convergence theorem completes the proof.
\end{proof}

Now we are ready to state and prove the main theorem of this section. We assume that $n\geq2$, the $n=1$ case will be proved in Theorem \ref{thm:WpSg}.

\begin{theorem}\label{thm:WpTd}
Let $n\geq2$ and $p\geq1$. Then the $p$-Wasserstein torus $\Wp(\Tn)$ is isometrically rigid, that is, the push-forward operation $\#:\isom(\Tn)\to\isom(\Wp(\Tn))$ is surjective.
\end{theorem}

\begin{proof}
Assume that $\Phi\colon\WpTd\to\WpTd$ is an isometry. We have to show that there exists an isometry $\psi:\Tn\to\Tn$ such that
\begin{equation}\label{eq:WpTd}
	\Phi(\mu) = \psi_\#\mu \qquad (\mu\in\WpSg).
\end{equation}

We already know that \eqref{eq:WpTd} holds for Dirac masses with some $\psi\in\isom(\Tn)$, which implies that $\ler{\psi^{-1}}_\#\circ\Phi$ fixes all Dirac measures. Since $\Phi=\psi_\#$ if and only if $\ler{\psi^{-1}}_\#\circ\Phi$ is the identity of $\Wp(\Tn)$, we can assume without loss of generality that $\Phi$ itself fixes all Dirac measures, and our task now is to prove that $\Phi(\mu)=\mu$ for all $\mu\in\Wp(\Tn)$. In fact, since $\Phi$ is continuous, it is enough to show that $\Phi$ fixes a dense subset of probability measures. We consider three different cases, corresponding to Lemma \ref{lem:WpTd-pot-p}.\\

\emph{Case (a) -- When $p=1$ holds.}
According to the density argument above, it suffices to prove that measures of the following form are fixed by $\Phi$: 
$\mu = \sum_{k=1}^N w_k\delta_{x^k} \in\WoTd$ where $N\in\N$, $\sum_{k=1}^N w_k = 1$, $w_k\geq 0$ for all $k$, and $x^1,x^2,\dots,x^N$ are pair-wise different points such that for all $k\in\{1,2,\dots,N\}$ we have 
$$ 
\{x^1,x^2,\dots,x^N\} \cap \ler{\bigcup_{j=1}^n H\ler{x^k+\frac{1}{2}e^j,j}} = \emptyset. 
$$
For such a measure we have the following for every $k$:
\begin{align*}
w_k &= \mu(\{x^k\}) = \lim_{s\to 0+}\frac{\potmu(x^k+s\cdot e^j) - 2\potmu(x^k) + \potmu(x^k-s\cdot e^j)}{2s} \\ 
& = \lim_{s\to 0+}\frac{\potphimu(x^k+s\cdot e^j) - 2\potphimu(x^k) + \potphimu(x^k-s\cdot e^j)}{2s} \leq \Phi(\mu)(\{x^k\}),
\end{align*}
where we used \eqref{eq:WpTd-pot-p=1} in the last step.
Since $1 = \sum_{k=1}^N \mu(\{x^k\}) \leq \sum_{k=1}^N \Phi(\mu)(\{x^k\}) \leq 1$, we obtain that $\Phi(\mu) = \mu$, which completes the proof of the present case.\\

\emph{Case (b) -- When either $1<p<2$ or $p>2$ holds.}
Assume first that $1<p<2$. By a density argument we see that it suffices to show that $\Phi$ fixes every finitely supported measure $\mu = \sum_{k=1}^N w_k\delta_{x^k} \in\WoTd$ where $x^1,x^2,\dots,x^N\in\Tn$ are points whose first coordinates are pairwise different. By Lemma \ref{lem:WpTd-pot-p} we have for all $k$ that 
\begin{align*}
    w_k &= \mu(\{x^k\}) = 4^{\frac{p-2}{2}}\int_{H(x^k,1)} \ler{\frac{1}{4} + \varrho_{n-1}\ler{\check{x}_1,\check{y}_1}^2}^{\frac{p-2}{2}} ~\dd\mu(y) \\ 
    & = \lim_{s\to 0+}\frac{\potmu(x^k+\frac{1}{2}e^1+s\cdot e^1) - 2\potmu(x^k+\frac{1}{2}e^1) + \potmu(x^k+\frac{1}{2}e^1-s\cdot e^1)}{-p\cdot4^{\frac{2-p}{2}}\cdot s} \\ 
    & = \lim_{s\to 0+}\frac{\potphimu(x^k+\frac{1}{2}e^1+s\cdot e^1) - 2\potphimu(x^k+\frac{1}{2}e^1) + \potphimu(x^k+\frac{1}{2}e^1-s\cdot e^1)}{-p\cdot4^{\frac{2-p}{2}}\cdot s} \\
    & = 4^{\frac{p-2}{2}}\int_{H(x^k,1)} \ler{\frac{1}{4} + \varrho_{n-1}\ler{\check{x}_1,\check{y}_1}^2}^{\frac{p-2}{2}} d\Phi(\mu)(y) \leq \Phi(\mu)(H(x^k,1))
\end{align*}
where we have equation if and only if $\Phi(\mu)(H(x^k,1)) = \Phi(\mu)(\{x^k\})$. As 
$$1 = \sum_{k=1}^N w_k \leq \sum_{k=1}^N \Phi(\mu)(H(x^k,1))\leq 1,$$
we must have $\Phi(\mu)(H(x^k,1)) = \Phi(\mu)(\{x^k\})$, and hence $\Phi(\mu) = \mu$. The very same argument works if $p>2$.\\

\emph{Case (c) -- When $p=2$ holds.}
First notice that by Lemma \ref{lem:WpTd-pot-p} we have 
\begin{equation}\label{eq:hyperplane}
    \Phi(\mu)\ler{H(z,j)} = \mu\ler{H(z,j)} \qquad (z\in\Tn, \; j=1,\dots,n),
\end{equation}
which implies that $\Phi$ preserves the one-dimensional marginals. That is,
\begin{equation}\label{eq:marginals}
    \ler{\proj_j}_\#\Phi(\mu) = \ler{\proj_j}_\#\mu \qquad (j\in\{1,\dots,n\}, \mu\in\WtTd),
\end{equation}
where $\proj_j\colon \T^n\to\T$, $\proj_j(x)=x_j$ is the projection map. Indeed, \eqref{eq:hyperplane} implies this if $\mu$ is supported on a finite set, and we obtain \eqref{eq:marginals} for general measures by a simple continuity argument.

We claim that measures supported on two points are left invariant by $\Phi$. Let us consider a measure $\mu = \alpha\delta_x+(1-\alpha)\delta_y$ where $x\neq y$ and $0<\alpha<1$. Without loss of generality we can assume that the representing vectors' coordinates satisfy the inequalities $-1/2\leq y_j-x_j<1/2$ for all $j$. Consider the following subset of $\Tn$:
$$
\cube := \big\{u\in\Tn\,|\,\forall j\in\{1,\dots,n\}:\, 0\leq \epsilon_j\cdot(u_j - x_j) \leq 1/2\big\},
$$
where for each $j$ we choose $\epsilon_j = 1$ if $x_j\leq y_j$, and $\epsilon_j = -1$ if $x_j>y_j$. By definition, $x,y\in\cube$. Note that by \eqref{eq:marginals}, we have \begin{equation}\label{eq:supp-phimu}
    S(\Phi(\mu))\subset\big\{z\in\Tn\;|\; \forall \;j\in\{1,\dots,n\}:\, z_j\in\{x_j,y_j\}\big\} \subset\cube.
\end{equation}
Furthermore, the subset $\cube$ as a metric space is isometrically isomorphic to the cube $[0,1/2]^n$ equipped with the usual Euclidean distance. Therefore the set of all probability measures supported on $\cube$, which we denote by $\Wt(\cube)$, can be considered as a subset of $\Wt(\R^n)$. 

Note that \eqref{eq:marginals} implies that $\Phi$ maps $\Wt(\cube)$ onto itself. Denote by $\Phi_\cube$ the restricted isometry $\Phi|_{\Wt(\cube)}\colon\Wt(\cube)\to\Wt(\cube)$. Define the centre of mass of any $\mu\in\Wt(\cube)$ as the unique point $m(\mu)\in\cube$ such that $\dwt\ler{\delta_{m(\mu)},\mu} = \dwt\ler{\{\delta_z\colon z\in\cube\},\mu}$, and the standard deviation of $\mu$ as the distance $\sigma(\mu) := \dwt\ler{\delta_{m(\mu)},\mu}\in[0,\infty)$. Since $\Phi_\cube$ leaves every Dirac mass invariant, we obtain that $\Phi_\cube$ preserves the centre of mass an standard deviation of measures, that is,
$$m(\Phi_\cube(\mu)) = m(\mu),\;\;\; \sigma(\Phi_\cube(\mu)) = \sigma(\mu)\qquad (\mu\in\Wt(\cube)).$$
Therefore, according to \cite[Lemma 6.2]{Kloeckner-2010}, for all $\mu,\nu\in\Wt(\cube)$ the following equivalence holds: $\langle v_1-v_2, w_1-w_2\rangle = 0$ for all $v_1,v_2\in S(\mu)$, $w_1,w_2\in S(\mu)$ if and only if the same holds for all $v_1,v_2\in S(\Phi_\cube(\mu))$, $w_1,w_2\in S(\Phi_\cube(\mu))$.

On the one hand, if $n\geq 3$ and $x_j\neq y_j$ for all $j\in\{1,\dots,n\}$, then there exists a $\nu := \alpha\delta_z+(1-\alpha)\delta_u\in\Wt(\cube)$ ($z\neq u$), such that the points $x,y,z,u\in\cube$ satisfy the following conditions:
$\langle x-y,z-u\rangle = 0$, but $\langle \xi^1-\xi^2, \zeta^1-\zeta^2 \rangle \neq 0$ holds for all other 
$$\xi^1,\xi^2\in\big\{\xi\in[0,1/2]^n\,|\, \forall\;j\in\{1,\dots,n\}:\, \xi_j\in\{x_j,y_j\}\big\},~~~\xi^1\neq\xi^2$$ and $$\zeta^1,\zeta^2\in\big\{\zeta\in[0,1/2]^n\,|\, \forall\;j\in\{1,\dots,n\}:\, \zeta_j\in\{z_j,u_j\}\big\}, ~~~\zeta^1\neq\zeta^2.$$ 
Note that then
\begin{equation}\label{eq:supp-phinu}
    S(\Phi(\nu))\subseteq\big\{v\in\Tn\,|\, \forall \;j\in\{1\dots,n\}:\, v_j\in\{z_j,u_j\}\big\} \subset\cube.
\end{equation}
Since the supports of $\mu$ and $\nu$ are perpendicular to each other, so must be the supports of their images. However, by our assumptions and \eqref{eq:supp-phimu}--\eqref{eq:supp-phinu} imply $\Phi(\mu) = \mu$, $\Phi(\nu) = \nu$. This, together with a simple continuity argument proves that indeed $\Phi$ leaves every measure fixed that are supported on at most two points, provided that $n\geq 3$.

On the other hand, if $n=2$, then consider a $\nu := \alpha\delta_{z}+(1-\alpha)\delta_{u}\in\Wt(\cube)$ such that $\langle x-y, z-u \rangle = 0$. Elementary geometric observation then gives that either $S(\Phi(\mu)) = \{(x_1,x_2),(y_1,y_2)\}$ and $S(\Phi(\nu)) = \{(z_1,z_2),(u_1,u_2)\}$, or $S(\Phi(\mu)) = \{(x_1,y_2),(y_1,x_2)\}$ and $S(\Phi(\nu)) = \{(z_1,u_2),(u_1,z_2)\}$. By \eqref{eq:marginals}, the latter cannot happen unless $\alpha=1/2$. Therefore, by continuity, $\Phi$ leaves every measure fixed that are supported on at most two points, also in this case.

From here it suffices to show that any finitely supported measure is left fixed by $\Phi$. Consider a $\mu = \sum_{k=1}^N w_k\delta_{x^k} \in\WoTd$ where $N\in\N$ and $x^1,x^2,\dots,x^N$ are pair-wise different points. Define the finite set
$$
F:=\big\{\xi\in\Tn\,|\, \forall\;j\in\{1,\dots,n\}:\, \xi_j\in\{x^1_j,x^2_j,\dots,x^N_j\}\big\}.
$$
By \eqref{eq:marginals}, we have $S(\Phi(\mu))\subseteq F$. Consider an arbitrary element $u\in F$ and observe that there exists two points $x,y\in\Tn$ such that $B(x,y)\cap F = \{u\}$. Note that since $\mu(B(x,y)) = \mu(\{u\})$, it is enough to show that $\Phi(\mu)(B(x,y)) = \mu(B(x,y))$. The latter is a consequence of the following equivalence which holds for all $\eta\in\WpTd$, $\alpha\in[0,1]$:
\begin{align*}
    &\dwt\ler{\eta,\{a\delta_x+(1-a)\delta_y\colon0\leq a\leq 1\}} = \dwt\ler{\eta,\alpha\delta_x+(1-\alpha)\delta_y} \\
    &\iff \mu\ler{\{z\colon \varrho(x,z) < \varrho(y,z)\}} \leq \alpha \leq \mu\ler{\{z\colon \varrho(x,z) \geq \varrho(y,z)\}}.
\end{align*}
This concludes the proof.
\end{proof}

\section{Isometric rigidity of the Wasserstein sphere} Similarly to the case of the torus, we prove first that if $\Phi$ is an isometry of $\WpSg$ then there exists an isometry $\psi\colon \Sn\to \Sn$ such that
$$
\Phi(\delta_x) = \psi_\#\delta_x = \delta_{\psi(x)} \quad (x\in\Sn).
$$
The diameter of $\WpSg$ is $\pi$ and this maximal distance is achieved if and only if the two measures in question are Dirac masses concentrated on antipodal points. This property must be preserved by isometries, and therefore the image of every Dirac measure is a Dirac measure again. Furthermore, since $\sphericalangle(x,y)=\dwp(\delta_x,\delta_y)$, the map $\psi$ defined via $\Phi(\delta_x):=\delta_{\psi(x)}$ is an isometry. Note that every $\psi\in\isom(\Sn)$ is the restriction of  an orthogonal transformation of the underlying space. Again, we are going to use the Wasserstein potential of the measure $\mu\in\WpSg$
\begin{equation}\label{eq:WpSg-pot}
	\potmu\colon \Sn\to \R, \;\; x\mapsto \dwpp(\mu,\delta_x)=
    \int_{\Sn} \sphericalangle(x,y)^p ~\dd\mu(y). 
\end{equation}
Since $\Phi(\delta_x) = \delta_{\psi(x)}$ for all $x\in\Sn$, we have $\potphimu\circ\psi=\potmu$. Indeed, 
\begin{equation}\label{eq:WpSg-phipot}
    \potphimu(\psi(x))=\dwpp(\Phi(\mu),\delta_{\psi(x)})=\dwpp(\Phi(\mu),\Phi(\delta_x) = \dwpp(\mu,\delta_x)=\potmu(x)
\end{equation}
for all $x\in\Sn$. The following is an analogue of Lemma \ref{lem:WpTd-pot-p}.
\begin{lemma}\label{lem:WpSg-pot-p>=1}
Let $n\in\N$, $p\geq 1$, $x,z\in\Sn$, $\sphericalangle(x,z) = \pi/2$, and $\mu\in\WpSg$. We have
\begin{align}\label{eq:WpSg-pot-p>=1}
    &\lim_{s\to 0+}\frac{\potmu(\cos s \cdot x + \sin s \cdot z) - 2\potmu(x) + \potmu(\cos s \cdot x - \sin s \cdot z)}{s}\nonumber \\
    & \hspace{1.5cm} = \left\{\begin{array}{cc}
    -2 \cdot \mu(\{-x\}) + 2 \cdot \mu(\{x\}), & \text{if}\; p=1 \\
    -2p\pi^{p-1} \cdot \mu(\{-x\}), & \text{if}\; p>1 
\end{array}\right..
\end{align}
\end{lemma}

\begin{proof}
Since the left hand-side of \eqref{eq:WpSg-pot-p>=1} is
\begin{align*}
    & \lim_{s\to0+} \bigg( \mu(\{x\})\frac{2\cdot s^p}{s} + \mu(\{-x\})\frac{2\cdot(\pi-s)^p - 2\cdot\pi^p}{s} \\
    &\hspace{0.3cm}+ \int_{\Sn\setminus\{-x,x\}} \frac{\sphericalangle(\cos s \cdot x + \sin s \cdot z,y)^p - 2\sphericalangle(x,y)^p + \sphericalangle(\cos s \cdot x - \sin s \cdot z,y)^p}{s} ~\dd\mu(y) \bigg),
\end{align*}
it suffices to show that the limit of the above integral is zero. Note that the function $t\mapsto t^p$ is Lipschitz on the interval $[0,\pi]$ with a constant, say, $K>0$. Hence the integrand is bounded, as can be seen by the following estimation (we use the triangle inequality in the last step)
\begin{align*}
    &\left| \frac{\sphericalangle(\cos s \cdot x \pm \sin s \cdot z,y)^p - \sphericalangle(x,y)^p}{s} \right| \leq K\cdot\left| \frac{\sphericalangle(\cos s \cdot x \pm \sin s \cdot z,y) - \sphericalangle(x,y)}{s} \right| \\
    & = K\cdot\left| \frac{\sphericalangle(\cos s \cdot x \pm \sin s \cdot z,y) - \sphericalangle(x,y)}{\sphericalangle(\cos s \cdot x \pm \sin s \cdot z,x)} \right| \leq K.
\end{align*}
Observe that for all $y\in\Sn\setminus\{-x,x\}$ the function 
\begin{equation}\label{eq:diffhato}
    s\mapsto \sphericalangle(\cos s \cdot x + \sin s \cdot z,y)^p = \arccos^p\inner{\cos s \cdot x + \sin s \cdot z}{y}
\end{equation}
is differentiable at $s=0$. Hence the point-wise limit of the integrand is the constant $0$ function. Applying the Lebesgue dominant convergence theorem finishes the proof.
\end{proof}

Using Lemma \ref{lem:WpSg-pot-p>=1} we can prove the main result of this section, namely that the $p$-Wasserstein sphere is isometrically rigid for all $p\geq1$. Since $\T$ can be identified with $\mathbb{S}^1$, this theorem completes the case of the torus as well.
\begin{theorem}\label{thm:WpSg}
Let $n\geq1$ and $p\geq1$. Then the $p$-Wasserstein sphere $\Wp(\Sn)$ is isometrically rigid, that is, the push-forward operation: $\#:\isom(\Sn)\to\isom(\Wp(\Sn))$ is surjective.
\end{theorem}

\begin{proof}
Assume that $\Phi\colon\WpSg\to\WpSg$ is an isometry. We have to show that there exists an isometry $\psi\in\isom(\Sn)$ such that
\begin{equation}\label{eq:WpSg}
	\Phi(\mu) = \psi_\#\mu \qquad (\mu\in\WpSg).
\end{equation}
We know that \eqref{eq:WpSg} holds for all Dirac masses with some isometry $\psi$. It suffices to show that \eqref{eq:WpSg} holds also for measures whose support is a finite set not containing any pair of antipodal points, as these form a dense subset of $\WpSg$. Consider a measure $\mu\in\WpSg$ with such properties, say, $\mu = \sum_{j=1}^N w_j\delta_{x_j}$ with $N\in\N$, $\sum_{j=1}^N w_j=1$, $\{x_1,\dots,x_N\}\cap\{-x_1,\dots,-x_N\}=\emptyset$ and $x_1,\dots,x_N$ pair-wise different. From here we distinguish between three cases.

First, if $p> 1$, then by Lemma \ref{lem:WpSg-pot-p>=1} we infer the following for all $j\in\{1,\dots,N\}$:
\begin{align*}
    w_j &= \mu(\{x_j\})= \lim_{s\to 0+}\frac{1}{-2p\pi^{p-1}s} \bigg(\potmu(\cos s \cdot (-x_j) + \sin s \cdot z_j) \nonumber\\
    &\hspace{6.1cm} - 2\potmu(-x_j) + \potmu(\cos s \cdot (-x_j) - \sin s \cdot z_j)\bigg) \nonumber\\
    &= \lim_{s\to 0+}\frac{1}{-2p\pi^{p-1}s} \bigg(\potphimu(\cos s \cdot (-\psi(x_j)) + \sin s \cdot \psi(z_j)) - 2\potphimu(-\psi(x_j)) \nonumber\\
    &\hspace{6.1cm}+ \potphimu(\cos s \cdot (-\psi(x_j)) - \sin s \cdot \psi(z_j))\bigg) \nonumber\\
    &= \Phi(\mu)(\{\psi(x_j)\})
\end{align*}
where we used \eqref{eq:WpSg-phipot} and that $\psi$ is the restriction of a linear isometry of the underlying real Hilbert space. This implies \eqref{eq:WpSg} for $\mu$ and completes the proof of this case.

Second, if $p=1$, then using the same calculation as above, we arrive at 
$$ w_j = \mu(\{x_j\}) = \Phi(\mu)(\{\psi(x_j)\}) - \Phi(\mu)(\{-\psi(x_j)\}).$$
Since $1 = \sum_{j=1}^N w_j \leq \sum_{j=1}^N \Phi(\mu)(\{\psi(x_j)\}) \leq 1$, we must have $\Phi(\mu)(\{-\psi(x_j)\}) = 0$ for all $j$, hence this case is done too.
\end{proof}

\end{document}